\documentclass[leqno,11pt,a4paper]{article}
\usepackage{amsmath,amssymb,amsthm,mathrsfs,dsfont}

\numberwithin{equation}{section}

\newtheorem {theorem}{Theorem}[section]
\newtheorem {prop}[theorem]{Proposition}
\newtheorem {lem}[theorem]{Lemma}
\newtheorem {cor}[theorem]{Corollary}
\newtheorem{ass}{Claim}[section]

{\theoremstyle{definition}
\newtheorem* {defi}{Definition}
}
{\theoremstyle{remark}
\newtheorem {remark}{Remark}[section]
\newtheorem {example}{Example}[section]
}

\def\ba{\begin{array}}
\def\ea{\end{array}}
\def\be{\begin{equation} \label}
\def\ee{\end{equation}}
\def\bit{\begin{itemize}}
\def\eit{\end{itemize}}
\def\ben{\begin{enumerate}}
\def\een{\end{enumerate}}

\def\lo{\mathopen{]}\/}  
\def\ro{\/\mathclose{[}} 

\def\RR{\mathbb{R}}
\def\NN{\mathbb{N}}
\def\ZZ{\mathbb{Z}}

\def\a{\alpha}
\def\b{\beta}
\def\g{\gamma}
\def\d{\delta}

\def\z{\zeta}
\def\n{\eta}
\def\th{\vartheta}
\def\k{\kappa}
\def\l{\lambda}
\def\r{\varrho}
\def\s{\sigma}

\def\ph{\varphi}
\def\o{\omega}
\def\G{\Gamma}
\def\D{\Delta}
\def\Th{\Theta}
\def\L{\Lambda}
\def\S{\Sigma}
\def\O{\Omega}

\def\sE{\mathsf{E}}
\def\sM{\mathsf{M}}

\def\cA{\mathcal{A}}

\def\cF{\mathcal{F}}
\def\cL{\mathcal{L}}
\def\cS{\mathcal{S}}

\def\sG{\mathscr{G}}
\def\sP{\mathscr{P}}

\def\inv{^{-1}}
\def\ti{\to\infty}
\def\1{\mathds{1}}
\def\pr{\text{\rm pr}}
\newcommand{\diam}{\mathrm{diam}}

\def\HE{\mathcal{E}}
\def\out{{\overline{\o}}}
\def\Ocr{\O_\text{\rm cr}}
\def\hOcr{\hat\O_\text{\rm cr}}
\def\bsub{\Subset\RR^d}
\def\Gb{\overline{\G}}
\newcommand{\bout}[2][\L]{\partial_{#1}#2}
\newcommand{\boutS}[3][\L]{\partial_{#1}^{#2}{#3}}

\newcommand{\Del}{\mathsf{Del}}
\newcommand{\Vor}{\mathsf{Vor}}

\newcommand{\LC}{\mathsf{LC}}
\newcommand{\Gab}{\mathsf{Gab}}
\newcommand{\SG}{{\mathsf{SG}}}

\begin{document}

\title{\bfseries Existence of Gibbsian point processes\\with geometry-dependent interactions}

\author{David Dereudre\footnotemark[1], Remy Drouilhet\footnotemark[2] \ and Hans-Otto Georgii\footnotemark[3]}
\date{\today}
\maketitle

\begin{abstract} \noindent
We establish the existence of stationary Gibbsian point processes for interactions that act on hyperedges
between the points. For example, such interactions can depend on Delaunay edges or triangles, 
cliques of Voronoi cells or clusters of $k$-nearest neighbors. The classical case of pair interactions is also included.
The basic tools are an entropy bound and stationarity.


\end{abstract}

\footnotetext[1]{LAMAV, Universit\'e de Valenciennes et du Hainaut-Cambr\'esis, Le Mont Houy
 59313 Valenciennes Cedex 09, France. 
E-mail: david.dereudre@univ-valenciennes.fr}

\footnotetext[2]{LJK, Universit\'e de Grenoble, B.S.H.M., 1251 Av. Centrale, BP 47 38040 Grenoble Cedex 9, France.
E-mail: Remy.Drouilhet@upmf-grenoble.fr}

\footnotetext[3]{Mathematisches Institut der Universit\"{a}t
         M\"{u}nchen, Theresienstra{\ss}e 39, 80333 M\"{u}nchen, Germany.
         E-mail: georgii@math.lmu.de}

\section{Introduction}

Recent developments in statistical physics, stochastic geometry and spatial statistics involve Gibbs point processes with interactions depending on the local geometry of configurations in $\RR^d$. A prominent class of such interactions is based on the nearest-neighbor graph coming from the Delaunay triangulation. In biology, such systems are used to model interacting cells in tissues or foams \cite{FRAEJ}, \cite{LS}. In spatial statistics and stochastic geometry, structured point patterns, point processes lying along fibers or regular Delaunay tessellations have been studied via geometric Gibbs modifications; see \cite{BBD99a}, \cite{Dereudre} and \cite{Ripley}. A probabilistic motivation comes from stationary renewal processes:
Since these can be characterized as Gibbs processes for interactions between nearest-neighbor
pairs of points \cite[Section 6]{HS}, Gibbs processes on $\RR^d$ with Delaunay tile interaction can be viewed as a
multi-dimensional counterpart \cite{DG}.  

In this paper we consider general geometry-dependent interactions that are defined on a hypergraph structure $\HE$. For every point configuration $\o$ in~$\RR^d$, $\HE(\o)$ denotes a set of hyperedges on $\o$,
and the formal Hamiltonian of $\o$ is given by 
$$ H(\o)=\sum_{\n\in\HE(\o)}  \ph(\n,\o)$$
with a potential $\ph(\n,\o)$  which only depends on $\n$ and the points of $\o$
in some neighborhood of $\n$. This locality of $\ph$ will be called the finite horizon property, see \eqref{eq:fhp} below. 
This general setting includes all the above-mentioned cases and also the classical many-body interactions of 
finite range that are familiar from statistical mechanics \cite{Ruelle}.

There is a principal difference between the geometric interactions considered here
and the classical many-body interactions. Namely, suppose
a particle configuration $\o$ is augmented by a new particle at $x$. In the case of a
many-body interaction, this is only influenced by an additional interaction term between $x$ and
the particles of $\o$, and the interaction between the particles of $\o$ is not affected by $x$. 
In other words, the classical many-body interactions are additive.
In our setting, the new particle at $x$ typically alters the hyperedges around $x$
completely: some hyperedges are created and some others are destroyed.  This means that both 
$\HE(\o\cup\{x\})\setminus\HE(\o)$ and  $\HE(\o)\setminus\HE(\o\cup\{x\})$ are non-empty,
so that $H(\o\cup\{x\})$ and $H(\o)$ each contain terms that are not present in the other.
This phenomenon blurs the usual distinction between attractive and repulsive interactions. 
Moreover, if the potential $\ph(\n,\o)$ is allowed to take the value~$\infty$ (hard-exclusion case), 
we arrive at the so-called non-hereditary situation that a configuration $\o$ is excluded although  $\o\cup\{x\}$ 
is possible. This last case makes it difficult to use an infinitesimal characterization of Gibbs measures in terms of their Campbell measures and Papangelou intensities. Nevertheless, such an infinitesimal approach was first used to prove the existence of Gibbs measures for Delaunay interactions by requiring geometric constraints on the interaction 
\cite{BBD99a,BBD99b,BBD08,Dereudre}. A quite different global approach, first used in \cite{GH} and based on stationarity and thermodynamic quantities such as pressure and free energy density, recently allowed to prove the existence of Gibbsian Delaunay tessellations for general bounded interactions without any geometric restrictions \cite{DG}, and a similar approach could also be applied to quermass-interaction processes~\cite{D09}.

In this paper we address the existence problem for the general formalism of hypergraph interactions introduced here. Our approach is global as in \cite{D09,DG} and leads to a significant improvement of the existing results. In particular, 
we establish the existence of Gibbsian Delaunay tessellations for non-bounded and  hard-exclusion potentials. In the classical context of stable many-body interaction of finite range, our results permit to relax the superstability assumption. 
Basic ingredients of the proof are an entropy bound to exploit the compactness of the level sets of the entropy density, 
and a somewhat delicate control of the range of the interaction, which takes advantage of stationarity.

The general setting of Gibbs measures for hypergraph interactions is introduced in Section~2. Section~3 contains our assumptions and the existence theorems. Section~4 offers a series of examples that includes many-body interactions of finite range as well as interactions on Delaunay tiles and Voronoi cells and between $k$-nearest neighbors. 
The proofs of the main results follow in Section~5, and an appendix is devoted to measurability questions.

\section{Preliminaries} 

\subsection{Point configurations and hyperedge interactions}

Consider the Euclidean space $\RR^d$ of arbitrary dimension $d\ge1$. Subregions of $\RR^d$ will typically be denoted by $\L$ or $\D$ and will always be assumed to be Borel with positive Lebesgue measure $|\L|$ resp.~$|\D|$.
We write $\D\bsub$ if $\D$ is bounded. 
A \emph{configuration} is  a subset $\o$ of $\RR^d$ which is locally finite, in that $\o\cap\D$ has finite cardinality
$N_\D(\o)=\#(\o\cap\D)$ for all $\D\bsub$. 
The space $\O$ of all configurations is equipped with the 
$\s$-algebra $\cF$ that is generated by the counting variables $N_\D$ with $\D\bsub$. 
It will often be convenient to write $\o_\D$ in place of $\o\cap\D$.
As usual, we take as reference measure on $(\O,\cF)$ the 
Poisson point process $\Pi^z$ of an arbitrary intensity $z>0$.  Recall that $\Pi^z$ is the unique probability measure
on $(\O,\cF)$ such that the following holds for all $\D\bsub$:
(i) $N_{\D}$ is Poisson distributed with parameter $z|\D|$, and (ii) conditional on $N_\D=n$, 
the $n$ points in $\D$ are independent with uniform distribution on $\D$, for each integer $n\ge1$.

Next, let $\O_f\subset\O$ denote the set of all finite configurations $\o$ (which means that $\#(\o)<\infty$), and
$\cF_f=\cF|_{\O_f}$ the trace $\s$-algebra of $\cF$ on $\O_f$. The product space $\O_f\times\O$
carries the product $\s$-algebra $\cF_f\otimes\cF$.
For each $\L\subset\RR^d$ we write $\O_\L=\{\o\in\O:\o\subset\L\}$
for the set of all configurations in $\L$, $\pr_\L:\o\to\o_\L=\o\cap\L$ for the projection
from $\O$ to $\O_\L$, $\cF_\L'= \cF|_{\O_\L}$ for the trace $\s$-algebra of $\cF$ on 
$\O_\L$, and $\cF_\L=\pr_\L\inv\cF_\L'\subset\cF$ for the $\s$-algebra of all events that happen in $\L$ only.
The reference measure on $(\O_\L,\cF_\L')$ is $\Pi^z_\L:=\Pi^z\circ\pr_\L\inv$.
Finally, let $\Th= (\th_x)_{x\in\RR^d}$ be the shift group, where $\th_x:\O\to\O$ is the translation by the vector 
$-x\in\RR^d$. By definition, $N_\D(\th_x\o)=N_{\D+x}(\o)$ for all $\D\bsub$.

The interaction of points to be considered in this paper will depend on the geometry of their location.
This geometry will be described in terms of a hypergraph, and the interaction potential will be defined on the
hyperedges.

\begin{defi}\mbox{}
\bit
\item A \emph{hypergraph structure} is a measurable subset $\HE$ of $\O_f\times\O$ such that  
$\n \subset \o$ for all $(\n,\o )\in\HE$.
If $(\n,\o )\in\HE$, we say that $\n$ is a \emph{hyperedge} of $\o$, and we write $\n\in\HE(\o)$.
 
\item A \emph{hyperedge potential} is a measurable function $\ph$ from a hypergraph structure 
$\HE$ to $\RR\cup\{\infty\}$. 

\item A hyperedge potential $\ph$ (or, more explicitly, the pair $(\HE,\ph)$) is called \emph{shift-invariant} if
\[
(\th_x\n,\th_x\o )\in\HE \text{ and } \ph(\th_x\n,\th_x\o )= \ph(\n,\o )
\text{ for all  $(\n,\o )\in\HE$ and $x\in \RR^d$.}
\]
\item Let us say that $\ph$ (or  the pair $(\HE,\ph)$) satisfies the \emph{finite horizon property}
if for each $(\n,\o )\in\HE$ there exists  some $\D\bsub$ such that 
\be{eq:fhp}
(\n,\tilde\o)\in\HE \text{ and }\ph(\n,\tilde\o)=\ph(\n,\o)\text{ when }\tilde\o=\o \text{ on }\D.
\ee
\eit
\end{defi}\noindent
\emph{We will assume throughout this paper that the hyperedge potential $\ph$ under consideration is shift-invariant 
and exhibits the finite horizon property. Moreover, for notational convenience we set $\ph=0$ on $\HE^c$}. Since $\HE$ is measurable,
this does not affect the measurability of $\ph$.

The domain $\HE$ of $\ph$ can be considered as a rule that turns each 
configuration $\o$ into a hypergraph $(\o,\HE(\o))$. Both $\HE$ and $\ph$ are not affected by translations.
Moreover, the presence of a hyperedge $\n\in\HE(\o)$ and the value of $\ph(\n,\o)$ can be determined
by looking at $\o$ in a (sufficiently large but) bounded 
neighborhood $\D$ of $\n$, called the \emph{horizon of $\n$ in $\o$}, which in general 
depends on both $\n$ and $\o$. Note that in general there is no minimal such horizon. 
To obtain a standard choice of  $\D$ one can take the closed ball $B_{\n,\o}=\bar B(g_\n,r_{\n,\o})$ with center
at the gravicenter $g_\n$ of $\n$ and radius $r_{\n,\o}$ chosen smallest possible. 
Finally, we note that the concept of hypergraph structure is similar to that of a
cluster property as introduced in \cite{Z}. 
Here are two examples the reader might keep in mind. 
Further examples will follow in Section~\ref{sec-examples}.

\begin{example}\label{ex:rangeR} \emph{Many-body interactions of bounded range}. Let $r>0$ and 
\[
\LC_r=\big\{(\n,\o): \n\subset\o,\  \text{diam}(\n)\le r,\ \o\in\O\big\}
\] 
be the locally complete graph. Thus, for each $\o\in\O$, $\LC_r(\o)$ consists of all hyperedges between points of distance at most~$r$. If we assume that $\ph(\n,\o)$ only depends on $\n$, we are in the classical situation of many-body interactions of range $r$. The finite horizon property for $(\n,\o)\in \LC_r$ then holds for arbitrary $\D\supset\n$.
If $\LC_r$ is restricted to hyperedges of cardinality two, we arrive at the familiar pair interactions of statistical mechanics.   
\end{example}

\begin{example}\label{ex:Delaunay} \emph{Delaunay potentials}.  The set $\Del$ of Delaunay hyperedges consists of all pairs 
$(\n,\o)$ with $\n\subset \o$ for which there exists an open ball $B(\n,\o)$ with $\partial B(\n,\o)\cap\o=\n$ that contains no points of $\o$. For $k=1,\ldots,d{+}1$ we write $\Del_k=\{(\n,\o)\in\Del: \#{\n}=k\}$ for the set of
all Delaunay simplices with $k$ vertices. Clearly, $\Del$ and $\Del_k$ are hypergraph structures.
It is possible that the convex hull of a set $\n\in\Del(\o)$ is not a simplex, namely when $\n$ consists
of four or more points on a sphere with no point inside. However, this is an exceptional case, which occurs only with probability zero for our Poisson reference measure~$\Pi^z$. 
Note that $B(\n,\o)$ is only uniquely determined when $\#\n=d{+}1$ and $\n$ is affinely independent.

The simplest class of Delaunay hyperedge potentials consists of pair interactions of the form
$\ph(\n,\o)=\phi(|x-y|)$ for $\n=\{x,y\}\in\Del_2(\o)$. Such a $\ph$ satisfies the finite horizon property 
\eqref{eq:fhp} with $\D=\bar B(\n,\o)$ for any ball $B(\n,\o)$ as above.

An example of a potential $\ph(\n,\o)$ on $\Del_2$ which does not only depend on $\n$ 
but also on $\o$ is  $\ph(\n,\o)=\phi(\Vor_{\o}(x),\Vor_{\o}(y))$. 
Here we write $\Vor_{\o}(x)$ for the Voronoi cell associated to a point $x\in\o$, viz.
the set 
\be{Voronoi}
\Vor_{\o}(x):=\big\{y\in\RR^d:\; |x-y| \leq |\tilde x-y| \quad \forall \,\tilde x\in\o\big\}
\ee
of all points of $\RR^d$ which are closer to $x$ than to all other points of $\o$. It is well-known that the
Voronoi cells form a tessellation \cite{SchnWeil}. Also, any two points of a configuration are connected by a
Delaunay edge if and only if their Voronoi cells have a non-trivial intersection. That is,
\be{delvor}
\{x,y\}\in \Del_2(\o)\Longleftrightarrow \#(\Vor_{\o}(x)\cap \Vor_{\o}(y))>1.
\ee
This reveals that the Delaunay graph is a nearest-neighbor graph.
The potential above\linebreak satisfies the finite horizon property \eqref{eq:fhp} with $\D$ equal to the closure of the set\linebreak 
$\bigcup_{\xi \in  \Del(\o),\; \xi\cap\n\neq \emptyset } B(\xi,\o)$, provided the cells $\Vor_{\o}(x)$ and $ \Vor_{\o}(y)$ are bounded. 
The proviso is necessary because unbounded Voronoi cells, which can occur at the ``boundary'' of $\o$, are not
protected by the points in $\D$. 
So we must exclude from $\Del_2(\o)$  all edges $\{x,y\}$ for which $\Vor_{\o}(x)\cup \Vor_{\o}(y)$ is unbounded.
\end{example}

\subsection{Gibbs measures for hyperedge potentials}

Our objective here is to introduce the concept of a Gibbsian point process for 
an activity $z>0$ and a hyperedge potential $\ph$
defined on a hypergraph structure $\HE$. First we will
introduce the Hamiltonian for a bounded region $\L\bsub$ with configurational boundary condition $\o\in\O$.
This requires to consider the set of hyperedges $\n$ in a configuration $\o$ for which either $\n$ itself or $\ph(\n,\o)$
depends on the points of $\o$ in $\L$. Specifically, we set
\be{eq:HE_L} 
\HE_{\L}(\o)=\big\{\n\in\HE(\o): \ph(\n,\z\cup\o_{\L^c})\ne\ph(\n,\o)\text{ for some } \z\in\O_\L
\big\}.
\ee
Recall the convention that $\ph=0$ on $\HE^c$. So, if $\n\in\HE(\o)\setminus\HE(\z\cup\o_{\L^c})$ for some
$\z\in\O_\L$ then either $\n\notin\HE_\L(\o)$, or $\n$ is irrelevant for $\o$, in that $\ph(\n,\o)=0$.
The \emph{Hamiltonian in $\L$ with boundary condition $\o$}  is then given 
by the formula
\be{localenergy}
H_{\L,\o}(\zeta):=\sum_{\n\in\HE_\L(\z\cup \o_{\L^c})} \ph(\n,\z\cup \o_{\L^c}) \quad\text{ for } \z\in\O_\L, 
\ee
provided this sum is well-defined. 
As usual, we also consider the associated partition function
\[
Z_{\L,\o}^z:=\int e^{- H_{\L,\o}(\zeta)}\, \Pi_\L^z(d\zeta)\,.
\]
To ensure that these quantities are well-defined, we
impose the following condition on the boundary condition $\o$. Let $\ph^-=(-\ph)\vee 0$ be the negative part of $\ph$.
 
\begin{defi}
A configuration $\o\in\O$ is called \emph{admissible for a region $\L\bsub$ 
and an activity $z>0$} if 
\[
H_{\L,\o}^-(\z):=\sum_{\n\in\HE_\L(\z\cup \o_{\L^c})} \ph^-(\n,\z\cup \o_{\L^c}) <\infty
\quad\text{for $\Pi_\L^z$-almost all $\z\in\O_\L$}
\]
(so that $H_{\L,\o}$ is almost surely well-defined),
and $0<Z_{\L,\o}^z<\infty$. We write $ \O_*^{\L,z}$ for the set of all these~$\o$.
\end{defi}
We note that, 
in contrast to the standard setting of statistical mechanics, the partition function is not automatically 
positive because, in the present setting,  $H_{\L,\o}(\emptyset)$ is not necessarily finite.
For $\o\in\O_*^{\L,z}$, we can define the Gibbs distribution for $(\HE,\ph,z)$ in a region $\L\bsub$ 
with boundary condition $\o$ as usual by
\be{eq:Gibbs_distr}
G^z_{\L,\o}(F)=\int_{\O_\L} \1_F(\z\cup\o_{\L^c})\, e^{-H_{\L,\o}(\z)}\, \Pi_\L^z(d\z)
\big/Z_{\L,\o}^z\,,
\ee
where $F\in\cF$ is arbitrary.

\begin{defi}
Let $\HE$ be a hypergraph structure, $\ph$ a hyperedge potential, and $z>0$ an activity.
A probability measure $P$ on $(\O,\cF)$ is called a \emph{Gibbs measure for $\HE$, $\ph$ and $z$} 
if $P(\O_*^{\L,z})=1$ and 
\be{DLR}
\int f\,dP = \int_{\O_*^{\L,z}} \frac1{Z_{\L,\o}^z}\int_{\O_\L} f(\zeta \cup \o_{\L^c}) \, e^{- H_{\L,\o}(\zeta)}\, 
\Pi_\L^z(d\zeta)\, P(d\o)
\ee
for every $\L\bsub$ and  every measurable $f:\O\to [0,\infty\ro$. 
\end{defi}
The equations (\ref{DLR}) are known as the DLR equations (after Dobrushin, Lanford and Ruelle).
They express that $G^z_{\L,\o}(F)$ is a version of the conditional probability $P(\,F\,|\cF_{\L^c})(\o)$. 
We will be particularly interested in Gibbs measures that are stationary, that is,
invariant under the shift group $\Th=(\th_x)_{x\in\RR^d}$.
We write $\sP_\Th$ for the set of all $\Th$-invariant probability measures $P$ on $(\O,\cF)$ with finite intensity 
$i(P)=\int N_{[0,1]^d}\,dP$, and $\sG_\Th(\ph,z)$ for the set of all Gibbs measures for $\ph$ and $z$ that belong to
$\sP_\Th$. 
We conclude this section with a discussion of measurability questions.

\begin{remark}\label{rem:measurability}\emph{Measurability}.
The quantities introduced above are not measurable with respect to the underlying $\s$-algebras
defined so far, but only with respect to their universal completion.
Specifically, for each $\s$-algebra $\cA$ let $\cA^*$ be the associated  $\s$-algebra of all
universally measurable sets, i.e., of the sets which belong to the $P$-completion of $\cA$
for all probability measures $P$ on $\cA$; see
\cite[pp.~36~\&~280]{Cohn}.
It is then the case that, for each $\L\bsub$,
the Hamiltonian $(\z,\o)\mapsto H_{\L,\o}(\z)$ is measurable with respect to
$(\cF_\L'\otimes\cF_{\L^c})^*$. Likewise, the partition function $\o\to Z_{\L,\o}^z$ is measurable with respect to
$\cF_{\L^c}^*$,  and $\O_*^{\L,z}\in\cF_{\L^c}^*$. Moreover,  $(\o,F)\to G^z_{\L,\o}(F)$ is a probability kernel
from $(\O_*^{\L,z},\cF_{\L^c}^*|_{\O_*^{\L,z}})$ to $(\O,\cF)$.
All this will be proved in the appendix.
\emph{We will therefore identify all probability measures in this paper 
with their respective complete extension.}
This convention underlies already the preceding definition of a Gibbs measure.
\end{remark}

\section{Hypotheses and results}

Having defined the concept of Gibbs measure for a hyperedge potential we now turn to our main theme,
the existence of such Gibbs measures. Let us state the conditions we need.
In the subsequent section we will provide a series of examples for which these conditions are met.

We begin with an assumption which says that hyperedges with a large horizon require the existence
of a large ball with only a few points. This will imply that the Hamiltonian $H_{\L,\o}$ depends only on the points
of $\o$ in a bounded region $\partial\L(\o)$,
and can be viewed as a sharpening of the finite horizon property \eqref{eq:fhp}. 
\begin{list}{}{\setlength{\leftmargin}{3em}\setlength{\labelwidth}{3em}}
\item[{\bf (R)}] \emph{The range condition}. There exist constants $\ell_R,n_R\in\NN$ and $\d_R<\infty$ such that for all 
$(\n,\o) \in \HE$ one can find a horizon $\D$ as in \eqref{eq:fhp} satisfying the following:\\
For every $x,y\in \D$, there exist $\ell$ open balls  $B_1,\ldots,B_\ell$ (with $\ell\le \ell_R$) such that \\
\hspace*{1em} -- the set $\cup_{i=1}^\ell \bar B_i$ is connected and contains $x$ and $y$, and\\
\hspace*{1em} -- for each $i$, either $\text{diam\,}B_i\le \d_R$ or $N_{B_i}(\o)\le n_R$.
\end{list}
Note that {\bf (R)} is trivially satisfied when all horizon sets can be chosen to have uniformly bounded diameters.
For instance, this holds in Example \ref{ex:rangeR}.
The use of the range condition {\bf (R)} will be revealed by Proposition~\ref{prop: range-confinement} below, 
which states that the following finite range property holds almost surely for nondegenerate $P\in\sP_\Th$.
 
\begin{defi}
Let $\L\bsub$ be given.
We say a configuration $\o\in\O$ \emph{confines the range of $\ph$ from $\L$} 
if there exists a set $\partial\L(\o)\bsub$ such that
$\ph(\n,\z\cup\tilde\o_{\L^c})=\ph(\n,\z\cup\o_{\L^c})$ whenever 
$\tilde\o=\o$ on $\partial\L(\o)$,  $\z\in\O_\L$ and $\n\in\HE_\L(\z\cup\o_{\L^c})$.
In this case we write $\o\in\Ocr^{\L}$.
$\partial\L(\o)$ is called the $\o$-boundary of $\L$, and we use the abbreviation
$\bout{\o}=\o_{\partial\L(\o)}$. 
\end{defi}
Given any $\o\in\Ocr^\L$, we assume in the following that $\partial\L(\o)=\L^r\setminus\L$, where $\L^r$ is the 
closed $r$-neighborhood of $\L$ and $r=r_{\L.\o}$ is chosen as small as possible. Moreover,
for $\o\in\Ocr^\L$ we have
\be{eq:Hfr}
H_{\L,\o}(\z)=\sum_{\n\in\HE_\L(\z\cup \bout{\o})} \ph(\n,\z\cup \bout{\o}),
\ee
and this sum extends over a finite set. This means that the first assumption in the definition of admissibility for $\L$ is satisfied.
Here is the proposition announced above. It will follow from Proposition~\ref{Ploc} below.

\begin{prop}\label{prop: range-confinement}Under {\bf (R)},
for each $\L\bsub$ there exists a set $\hOcr^\L\in\cF_{\L^c} $ such that $\hOcr^\L\subset\Ocr^\L$ and
$P(\hOcr^\L)=1$ for all $P\in\sP_\Th$ with $P(\{\emptyset\})=0$.
\end{prop}

\noindent 
Our next assumption is stability, the standard assumption that ensures the finiteness of all partition functions.
In our setting, a somewhat modified definition turns out to be suitable.

\begin{list}{}{\setlength{\leftmargin}{3em}\setlength{\labelwidth}{3em}}
\item[{\bf (S)}] \emph{Stability}. The hyperedge potential $\ph$ is called \emph{stable} if 
there exists a constant $c_S\ge 0$ such that 
\be{stability}
H_{\L,\o}(\zeta) \ge -c_S \;\#(\z\cup \bout{\o})
\ee
for all $\L\bsub$, $\z\in\O_\L$ and  $\o\in\Ocr^\L$. 
\end{list}
In Remark \ref{rem:stability} below we will show that this definition is a natural extension of the familiar concept
of stability in statistical mechanics. 
Complementary to the lower bound provided  by stability,  we will also need  a further condition that provides at least a partial
upper bound for the Hamiltonians.
This is because in the extreme case when $\ph$ is constantly  equal to $\infty$ we have $Z_{\L,\,\cdot}^z\equiv 0$,
so that the definition of Gibbs measures is meaningless. 

Let $\sM\in\RR^{d\times d}$ be an invertible $d\times d$ matrix and consider for each $k\in\ZZ^d$ the cell
\be{eq:cell}
C(k):=  \big\{\sM x\in\RR^d: x-k\in [-1/2,1/2[^d \big\}.
\ee  
These cells together constitute a periodic partition of $\RR^d$ into parallelotopes. For example, the columns
 $\sM_1,\ldots,\sM_d$ of $\sM$ might form an orthogonal basis of $\RR^d$ or, for $d=2$, define the sides 
of an equilateral triangle. For brevity we write $C=C(0)$.
Let $\G$ be a measurable subset of $\O_C\setminus \{\emptyset\}$ and 
\be{barS}
\Gb = \Big\{ \o\in\O : \th_{\sM k}(\o_{C(k)}) \in \G  \ \text{ for all }k\in\ZZ^d\Big\}
\ee
the set of all configurations whose restriction to an arbitrary cell $C(k)$, when shifted back to $C$, belongs to $\G$.
We call each $\o\in\Gb$ pseudo-periodic. The required control of the Hamiltonian 
from above will then be achieved by the following assumption on the joint behavior of $\HE$, $\ph$
and $z$. Recall that $r_{\L,\o}$ was defined before \eqref{eq:Hfr}.

\pagebreak[3]
\begin{list}{}{\setlength{\leftmargin}{3em}\setlength{\labelwidth}{3em}}
\item[{\bf(U)}]\emph{Upper regularity}. $\sM$ and $\G$ can be chosen so that
the following holds.\\[-2ex]
\begin{list}{}{\setlength{\leftmargin}{3em}\setlength{\labelwidth}{3em}}
\item[(U1)] \emph{Uniform confinement: } $\Gb\subset\Ocr^\L$ for all $\L\bsub$, and\\[.5ex] 
\hspace*{4em} ${\displaystyle r_\G:=\sup_{\L\bsub}\,\sup_{\o\in\Gb}\,r_{\L,\o}<\infty}$.
\item[(U2)] \emph{Uniform summability: }
$\displaystyle c_\G^+:= \sup_{\o\in\Gb}  \sum\limits_{\n\in\HE(\o):\,\n\cap C \neq \emptyset} 
\frac{\ph^+(\n,\o)}{\#(\hat\n)}<\infty$,\\[.5ex] 
where  $\hat\n:=\{k\in \ZZ^d: \n\cap C(k) \neq \emptyset \}$ and $\ph^+$ is the positive part of $\ph$.\\[-.7ex]
\item[(U3)] \emph{Strong non-rigidity: } $ e^{z |C|}\, \Pi^z_C(\G) > e^{c_{\G}}$, where $c_{\G}$ is defined as in
(U2) with $\ph$ in place of $\ph^+$.
\end{list}
\end{list}
Hypothesis (U1) states that the configurations in $\Gb$ confine the range of $\ph$ in a uniform way.
So, for $\o\in\Gb$, the $\o$-boundary $\partial\L(\o)$ of a set $\L\bsub$ is contained in the $r_\G$-boundary
$\partial^\G\L:= \L^{r_\G}\setminus\L$, and the cardinality of $\bout{\o}$ ist not larger than that of 
$\boutS{\G}{\o}:=\o_{\partial^\G\L}$.
On the other hand, condition (U2) provides a uniform upper bound  for the local Hamiltonians $H_{\L,\cdot}$  on 
$\Gb$. This bound is of order $c_\G |\L|$, cf.~\eqref{eq:H_upper_bound} below.
Finally, hypothesis (U3) holds for all $z$ in an interval of the form 
$\lo z_0,\infty\ro$, provided that $\Pi^z_C(\G)>0$ for some (and thus all) $z>0$. 
Indeed, since $\emptyset\notin\G$ it follows that
\be{interval}
e^{z|C|}\,\Pi_{C}^z(\G)= \sum_{k=1}^\infty \frac{z^k}{k!} \int_{C} \cdots \int_{C} 
 \1_{\G}(\{x_1,\ldots,x_k\})\, dx_1\ldots dx_k
\ee
is then a strictly increasing function of $z$. 
We emphasize that condition {\bf (U)} imposes conflicting demands on $\G$. While 
(U1) and (U2) suggest to choose the set $\G$ as small as possible, (U3) requires that $\G$ is not too small. 
The point will be to choose a set $\G$ that satisfies all requirements simultaneously. 
Here is our main existence theorem. 

\begin{theorem}\label{th1}
For every hypergraph structure $\HE$, hyperedge potential $\ph$ and activity $z>0$ satisfying
 {\bf (S)},  {\bf (R)} and  {\bf (U)} there exists at least one Gibbs measure $P\in\sG_\Th(\ph,z)$.
\end{theorem}

In some cases it is difficult to satisfy hypothesis (U3) when $z$ is small. 
This occurs, for example, when a hard-exclusion hyperedge interaction enforces a minimal number of points per cell; see Proposition~\ref{prop:del2rigid}.  But a slight variation of proof allows to establish  the existence of a Gibbs measure
for every $z>0$ also in this case. Assumptions  (U1) and (U3) 
are replaced by conditions (\^U1) and (\^U3) as follows. 
\begin{list}{}{\setlength{\leftmargin}{3em}\setlength{\labelwidth}{3em}}
\item[{\bf (\^U)}] \emph{Alternative upper regularity}. $\sM$ and $\G$ can be chosen so that
the following holds.\\[-2ex]
\begin{list}{}{\setlength{\leftmargin}{2.5em}\setlength{\labelwidth}{2.5em}}
\item[(\^U1)] \emph{Lower density bound: } There exist constants $a,b>0$ such that $ \#(\zeta) \ge a |\L|-b$
whenever  $\z\in\O_f$ is such that $H_{\L,\o}(\z)<\infty$ for some $\z\subset\L\bsub$ and some  $\o\in\Gb$.
\item[(\^U2)] = (U2) \emph{Uniform summability.}
\item[(\^U3)] \emph{Weak non-rigidity: } $\Pi^z_C(\G) > 0$.
\end{list}
\end{list}
Here is the modified existence theorem.

\begin{theorem}\label{th2}
A Gibbs measure $P\in\sG_\Th(\ph,z)$ exists also under the hypotheses {\bf (S)},  {\bf (R)} and  {\bf (\^U)}.
\end{theorem}

It is often natural to choose $\G$ as the set of all configurations that consist of a single point in some Borel
set $A\subset C$. So we define 
\be{SA}
\G^A =\big\{\z\in \O_C:  \z=\{x\} \text{ for some } x\in A \big\}.
\ee
For $\G=\G^A$, the assumptions {\bf (U)} and {\bf (\^U)} are respectively called  {\bf (U$^A$)} and {\bf (\^U$^A$)} . In particular, (U2$^A$) and (U3$^A$) take the simpler form\par
\hspace*{1ex}\parbox{10cm}{
\bit
\item[(U2$^A$)] $\displaystyle c_A^+:= \sup_{\o\in\Gb^A}  \sum\limits_{\n\in\HE(\o):\,\n\cap C \neq \emptyset}  
\frac{\ph^+(\n,\o)}{ \#(\n)}  <  \infty$ \quad and
\item[(U3$^A$)] $ z|A|\,>e^{c_A} $.
\eit}\par\noindent
We then have the following corollary of Theorems~\ref{th1} and \ref{th2}.
\begin{cor}\label{cor3}\label{cor4}
A Gibbs measure $P\in\sG_\Th(\ph,z)$ exists under the 
hypotheses {\bf (S)},  {\bf (R)} and either {\bf (U$^A$)} or {\bf (\^U$^A$)}.
\end{cor}

We conclude this section with a series of comments on our assumptions and on the extension
to marked particles. 

\begin{remark}\label{uniformlocal}  \emph{Bounded horizons}.
The conditions {\bf (R)} and {\bf (U$^A$)} hold as soon as  $\ph(\{0\},\{0\})$ is finite
and $(\HE,\ph)$ has bounded horizons, in that there exists some $r_\ph<\infty$ such that 
$r_{\n,\o}\le r_\ph$ for all $(\n,\o)\in\HE$. (The notation $r_{\n,\o}$ was
introduced after~\eqref{eq:fhp}.)
Indeed, condition {\bf (R)} holds trivially with $\d_R=2r_\ph$. As for {\bf (U$^A$)}, let $\sM=a\sE$, where $a>2r_\ph$ and $\sE$ is the identity matrix.
Let $A=B(0,b)$ be a centered ball of radius $b<a/2-r_\ph$. For $\G=\G^A$, condition (U1) holds with $r_\G=r_\ph$.
Moreover, by the choice of $a$ and $b$, each $\n\in\HE(\o)$ with $\o\in\Gb^A$ must be a singleton $\{x\}$,
so that $\ph(\n,\o)=\ph(\{x\},\{x\})$.
In view of the shift-invariance of $\ph$, this means that (U2$^A$) holds with $c_A^+=\ph^+(\{0\},\{0\})<\infty$.
Finally, (U3$^A$) holds if $a$ and $b$ are in fact chosen so large that also $\pi z b^2>e^{\ph(\{0\},\{0\})}$.
\end{remark}

\begin{remark}\label{rem:scale-invariant}\emph{Scale-invariant potentials}. Suppose $\HE$ and $\ph$ are scale-invariant 
in the sense that $(r\n,r\o)\in\HE$ and $\ph(r\n,r\o)=\ph(\n,\o)$ for all $(\n,\o)\in\HE$ and $r>0$. Here, $r\n=\{rx:x\in\n\}$ and $r\o=\{rx:x\in\o\}$. Then Theorem~\ref{th1} is still valid when assumption {(U3)} is replaced by {(\^U3)}. Indeed, the scale invariance of $\ph$ implies that the image of a Gibbs measure for $\ph$ and $z$ under the rescaling 
$\o\to r\o$ is a Gibbs measure for $\ph$ and $zr^{-d}$. So, it is sufficient to have the existence of a Gibbs measure for 
large $z$, and this follows from the remark around (\ref{interval}).
\end{remark}

\begin{remark} \label{rem:linearity}\emph{Stability via sublinearity of the hypergraph}.  We say that a hypergraph $\HE$ is sublinear if there exists a constant $C<\infty$ such that $\#\HE(\o)\le C\,\#(\o)$ for every \emph{finite} configuration $\o$. In this case, the stability is
ensured by requiring that the hyperedge potential $\ph$ is bounded below, in that 
\be{boundedbelow}
\ph(\n,\o) \ge - c_\ph
\ee
for some $c_\ph<\infty$. 
If the sublinearity of the hypergraph structure fails, the stability can simply be achieved by requiring that the potential $\ph$ is nonnegative (i.e., $c_\ph=0$). 
For example,  for $d=2$ it follows from Euler's formula that the cardinalities of the Delaunay edges and triangles are sublinear
\cite{DG}, so that the stability follows directly from~(\ref{boundedbelow}).
\end{remark}

\begin{remark}\label{rem:stability}\emph{Stability: comparison with the classical case}. 
Consider the hypergraph structure $\HE=\LC_r$ of Example \ref{ex:rangeR} describing many-body interactions of range $r$. In contrast to Example \ref{ex:Delaunay} where the Delaunay tiles depend on 
the presence of further particles, it is then meaningful to define the energy of a finite configuration $\z\in\O_f$ 
by $H(\z)=\sum_{\n\in\HE(\z)}\ph(\n)$. The classical stability condition asserts that $H(\z)\ge -c_S\,\#(\z)$
for all $\z\in\O_f$; see \cite{Ruelle}, for example. This follows from {\bf (S)} by choosing $\L\supset\z$ and $\o=\emptyset$.
Conversely, the Hamiltonian \eqref{localenergy} is equivalent to the Hamiltonian $\tilde H_{\L,\o}(\z):=H(\zeta\cup \bout{\o})$ in the sense that the associated Gibbs distributions coincide (at least when $\ph<\infty$), and the
classical stability assumption for $H$ gives {\bf (S)} for $\tilde H_{\L,\o}$. This shows that 
hypothesis {\bf (S)} is essentially equivalent to the classical concept of stability.
\end{remark}

\begin{remark}\label{rem:subgraph}\emph{Sub-hypergraph potentials}. Consider a shift-invariant sub-hyper\-graph 
structure $\HE' \subset \HE$
of $\HE$ and a hyperedge potential $\varphi$ on $\HE$,  and let $\varphi'$ be its restriction to $\HE'$. In general, 
$\varphi'$ does not satisfy the finite horizon property, but let us assume it does. 
Which of the assumptions {\bf (R)}, {\bf (S)}, {\bf (U)} and {\bf (\^U)} on $\ph$ are inherited by $\varphi'$? 
It is clear that assumptions {\bf (R)}, {(U1)}, {(U2)} and {(\^U3)} are hereditary. Assumption {(U3)} remains also valid 
for $\ph'$, but for a different range of values 
of $z$ because the constant $c_\G$ is different in general. Assumption {\bf (S)} is lost, but
a positive exception is the case of Remark \ref{rem:linearity} when stability follows from the sublinearity of $\HE$ and the lower boundedness of $\ph$; these properties are obviously inherited by $\varphi'$. Assumption {(\^U1)} is lost in general.  
\end{remark}

\begin{remark}\label{rem:U}\emph{Upper regularity in Delaunay models}.
For potentials acting on the Delaunay graph, the matrix $\sM$ and the set $\G$ in hypotheses {\bf (U$^A$)} 
and {\bf (\^U$^A$)} will be chosen as follows. Let $\sM$ be such that 
$|\sM_i|=a>0$ for $i=1,\ldots, d$ and $\sphericalangle(\sM_i,\sM_j)=\pi/3$ for $i\neq j$, 
and let $\G=\G^A$ with $A=B(0,b)$. If $b\le \r_0a$ for some sufficiently small constant $\r_0>0$, the Delaunay neighborhood of a point $x$ in a configuration $\o \in \Gb$ contains a minimal number of points denoted by $\g_d$. 
For $d=2$ one can take $\r_0=\sqrt{3}/6$ and has $\g_2=6$,  and the Delaunay neighborhood of the unique point 
$x_k$ in $\o\cap C(k)$ consists of the unique points $x_{l}$ in $\o\cap C(l)$ with $l-k\in\{(-1,0),(-1,1),(0,1),(1,0),(1,-1),(0,-1)\}$.
For $d>2$, $\g_d$ is less easy to determine but it is clearly not larger than $3^d-1$, the value  corresponding to the case  $\sM=a\sE$. 
\end{remark}

\begin{remark}\label{rem:marked}
\emph{Extension to the marked case}. 
The preceding results can be easily extended to the case of particles with internal degrees of freedom, or marks.
Let $\S$ be an arbitrary Polish space with Borel $\s$-algebra $\cS$ and reference probability measure~$\mu$. 
$\S$ serves as the space of marks. That is, each marked point is represented by a position $x\in\RR^d$
and a mark $\s\in\S$, and each configuration $\o$ is a countable subset of $\RR^d\times\S$ having a locally finite projection
onto $\RR^d$. The role of the reference measure on the configuration space $\O$ is taken over
by the Poisson point process $\Pi^z$ with intensity measure $z\l\otimes\mu$, where $\l$ is Lebesgue measure
on $\RR^d$.  The translations $\th_x$ act only on the positions of the particles and leave their marks untouched.  
We do not discuss the further formal details here, which are standard and can be found in \cite{KMM}
or \cite{GiiZess}, for example.
What we want to emphasize here is that all definitions and results above carry over to this setting without any change, provided it is understood
that all regions $\L$ or $\D$ in $\RR^d$ always refer to the positional part of a configuration. For example:
the notation $\o_\D$ now stands for $\o\cap(\D\times\S)$; 
a set $\D\bsub$ is the horizon of $(\n,\o)\in\HE$ if
$(\n,\tilde\o)\in\HE $ and $\ph(\n,\tilde\o)=\ph(\n,\o)$ whenever $\tilde\o=\o$  on $\D\times\S$;
and the condition in Remark \ref{uniformlocal}  should now read $\sup_{\s\in S}\ph(\{(0,\s),\{(0,\s)\})<\infty$
for some Borel set $S\subset\S$ with $\mu(S)>0$.
\end{remark}

\section{Examples}\label{sec-examples}

In this section we present a series of examples to illustrate our general existence results. These examples satisfy the assumptions of Theorem \ref{th1} or \ref{th2} and many of them have been introduced in practical or theoretical papers without justification or with justification in some partial cases. 

To sort the examples of this section we will distinguish whether or not the potential $\ph(\n,\o)$ depends explicitly on $\o$.
If not, we speak of a \emph{pure hyperedge potential}.
Otherwise, the finite horizon property implies that $\ph(\n,\o)$ actually depends only on some points of $\o\setminus\n$ close to $\n$, which is expressed by speaking of a \emph{neighborhood-dependent hyperedge potential}. 
But note that this distinction is merely a matter of how the interaction is represented. In the pure case,
the extended potential $(\n,\o)\to \1_{\HE}(\n,\o)\,\ph(\n)$ on $\O_f\times\O$ clearly does depend on $\o$, while in the other case one can often include the neighboring points of a hyperedge into an enlarged hyperedge to obtain a pure hyperedge potential.

Most of the following examples  are based on the Delaunay graph. For simplicity, we will then often confine ourselves
to the case $d=2$, in which the stability is ensured by Remark~\ref{rem:linearity} as soon as $\ph$ is bounded from below.
But the reader should note that analogous results hold also in higher dimensions when $\ph$ is nonnegative, so that
{\bf(S)} is trivial.

\subsection{Pure hyperedge interactions}

In this subsection we consider examples of hyperedge potentials $\ph$ which only depend on the first parameter,
so that $\ph(\n,\o)=\ph(\n,\n)=:\ph(\n)$.

\subsubsection{Many-body interactions of finite range}\label{sec:ruelle}

Let $r>0$ and $\HE=\LC_r$ be the locally complete graph of Example \ref{ex:rangeR}, and
suppose that $\ph(\n,\o)=\ph(\n)$.  Remark~\ref{uniformlocal} then shows that a Gibbs measure exists as soon as the potential $\ph$ is stable and  $\ph(\{0\})<\infty$. 
By Remark \ref{rem:stability}, the first condition is equivalent to the classical stability assumption, 
and the second is necessary for defining Gibbs measures.
So, as was observed first in \cite[Remark 4.2]{GH}, the techniques used here allow
to weaken the superstability assumption of Ruelle's classical existence result \cite{Ruelle70},
provided the interaction has finite range. But our techniques neither allow to treat the case of infinite range nor to
rederive Ruelle's probability estimates.

An example of a stable but not superstable many-body interaction is the so-called quermass interaction. In space dimension
$d=2$, the associated Hamiltonian for a configuration $\o$ is a linear combination of area, perimeter and Euler-Poincar\'e characteristic of the union of all discs of fixed radius that are centered at the points of $\o$.
By the additivity of the Minkowski functionals, this Hamiltonian can be expressed in terms of a many-body interaction of finite range.  Its stability is proved in \cite{KVB}, but the superstability fails. 
The existence of Gibbs processes for this interaction has been proved in \cite{D09} by the same methods used here.

\subsubsection{Delaunay edge interactions} \label{sec:del2}

Here we consider two classes of potentials $\ph$ on the hypergraph structure $\Del_2$ in $\RR^2$ 
which are bounded below  and depend only on the length of the Delaunay edges,
in that $\ph(\{x,y\}):=\phi(|x-y|)$ when $\{x,y\}\in\Del_2$. Such potentials have been studied 
in \cite{BBD99a,BBD99b,BBD99c,BBD08}. We improve the existence results of these papers.

\smallskip\noindent
\emph{\thesubsubsection.1  Polynomially increasing edge interactions}.
Suppose that
\be{del2:size}
\phi(\ell)\leq \k_0+\k_1 \ell^\a\quad \mbox{ for some constants $\k_0\geq 0$, $\k_1\geq 0$ and $\a>0$ } .
\ee
Recall the definition of $\r_0$ in Remark~\ref{rem:U}.

\begin{prop}\label{prop:del2size} Let $d=2$ and $\ph$ be a pure edge potential on $\Del_2$ which is bounded below
and satisfies \eqref{del2:size}. Then there exists at least one Gibbs measure for $\ph$ and every activity
 $$z>(1{+}2\r_0){e^{3\k_0}}({3}\a e^2 \k_1/{2})^{1/\a}/ (\pi \r_0^2) .$$
\end{prop}
\begin{proof}
We apply Corollary \ref{cor3} with the assumption {\bf (U$^A$)}. The stability condition {\bf (S)} holds 
because of Remark \ref{rem:linearity}. 
We know further from Example \ref{ex:Delaunay} that every edge $\n\in\Del_2(\o)$ has the finite horizon
$\bar B(\n,\o)$. This shows that assumption {\bf (R)} is satisfied. 
Concerning assumption {\bf (U$^A$)}, let $\sM$ and $\G^A$ be defined as in  Remark \ref{rem:U} 
with $b=\r_0a$ and $a$ to be chosen later. The assumptions (U1$^A$) and (U2$^A$) are then trivial. 
We also find that $c_A\le 3\big(\k_0+\k_1a^\a(1{+}2\r_0)^\a\big)$. So, (U3$^A$) holds as soon as
$ z>K_0\,e^{K_1a^\a}/{a^2}$,
where $K_0=e^{3\k_0}/(\pi\r_0^2)$ and $K_1=3\k_1(1{+}2\r_0)^\a$. Optimizing over $a$ we find the value 
$a=(2/\a K_1)^{{1}/{\a}}$, which gives the sufficient condition  $z>K_0(\a K_1e^2/2)^{{1}/{\a}}$.\qed\end{proof}

The bound for $z$ in Proposition~\ref{prop:del2size} is certainly not the best possible. However, in the bounded case $\k_1=0$ it reduces to the trivial condition $z>0$.
As for condition \eqref{del2:size}, at first sight it might seem surprising that $\phi(\ell)$ does not necessarily converge to zero when $\ell$ tends to infinity, but is even allowed to converge to $\infty$. However, $\phi(\ell)$ should not be compared with
a pairwise interaction that must decay over large distances, but with a nearest-neighbor interaction between oscillators
that form an elastic lattice. A potential like the harmonic interaction $\phi(\ell)=\ell^2$ is therefore quite natural.

\smallskip\noindent
\emph{\thesubsubsection.2  Long-edge exclusion}.
Suppose there are constants $0\le \ell_0<\ell_1 \le \ell_2$ such that
\be{del2:rigid}
 \sup_{\ell_0\le\ell\le \ell_1} \phi(\ell) <\infty  \quad \mbox{ and } \quad \phi(\ell)=\infty \ \mbox{ if $\ell>\ell_2$.}
\ee

\begin{prop}\label{prop:del2rigid}
Let $d=2$ and $\ph$ be a pure edge potential on $\Del_2$ which is bounded below
and satisfies \eqref{del2:rigid}. Then there exists at least one Gibbs measure for $\ph$ and every $z>0$.
\end{prop}
\begin{proof}  We apply Corollary \ref{cor3} with the assumption {\bf (\^U$^A$)}. 
Hypotheses \textbf{(S)} and {\bf (R)} follow as in Proposition~\ref{prop:del2size}.
Condition (\^U1$^A$) is satisfied since the long-edge exclusion condition
ensures a minimal density of points. 
The values $a$ and $b$ in the definition of $\sM$ and $\G$ are now chosen such that each $\{x,y\}\in\Del_2(\o)$ (with $\o\in\Gb^A$) satisfies $\ell_0<|x-y|<\ell_1$. Then  (\^U2$^A$) follows, and (\^U3$^A$) 
is obvious.\qed\end{proof} 

Since $\phi(\ell)$ may be equal to $\infty$ for $\ell$ less than some $r_0<\ell_0$, the present 
example includes the classical case of a hard-core interaction which forces the points to keep distance at least $r_0>0$ from each other.  

\begin{remark}\rm
Let us consider whether or not Propositions~\ref{prop:del2size} and \ref{prop:del2rigid} remain valid when the Delaunay graph 
is replaced by a subgraph, as was discussed in Remark \ref{rem:subgraph}. 
One possible subgraph is the Gabriel graph. By definition, 
the Gabriel edge set $\Gab_2(\o)$ consists of all edges $\{x,y\}\in \Del_2(\o)$ for which the open disc 
with center $(x+y)/2$ and radius $|x-y|/2$ contains no point of $\o$.
Since $\Gab_2$ is shift-invariant and local, the restriction of $\ph$ to $\Gab_2$ remains a shift-invariant hyperedge potential satisfying the finite horizon property \eqref{eq:fhp}.  By Remark~\ref{rem:linearity}, {\bf (S)} is inherited when we pass
from $\Del_2$ to $\Gab_2$. In view of our choice of $\sM$ and $\G$, the constant $c_\G$ remains unchanged, 
so that (U3) is also inherited without modification of the valid range for $z$. This means that Proposition~\ref{prop:del2size}
holds also for the Gabriel graph.
Unfortunately, condition (\^U1$^A$) is lost, so that Proposition~\ref{prop:del2rigid} does not carry over to the Gabriel graph. Another example of a Delaunay subgraph is the
minimum spanning tree graph. This is tailor-made to be non-local, so that our results cannot be applied.
\end{remark}

\subsubsection{Delaunay tile interactions}\label{sec:del3}

In this subsection we deal with potentials on the hypergraph structure $\Del_3$ of all Delaunay
triangles in $\RR^2$, and we still assume that $\ph$ is bounded below and depends only on the hyperedge and not on the remaining configuration. Such models have been considered recently in \cite{Dereudre} and \cite{DG}. We improve the existence results given there.

\smallskip\noindent
\emph{\thesubsubsection.1  Polynomially increasing triangle interactions}.
The first example is specified by assuming that the interaction of a Delaunay triangle $\n\in\Del_3$ is controlled by its size 
as in Proposition~\ref{prop:del2size}:
\be{del3:size}
\ph(\n)\leq \k_0+\k_1\, \d(\n)^\a\mbox{ for some constants $\k_0\geq 0$, $\k_1\geq 0$ and $\a>0$},
\ee
where $\d(\n)$ is the diameter of the circumcircle of $\n$.
We then have a similar result.

\begin{prop}\label{prop:del3size}
Let $d=2$ and $\ph$ be a pure triangle potential on $\Del_3$ which is bounded below
and satisfies \eqref{del3:size}. Then there exists at least one Gibbs measure for $\ph$ and every 
\[
 z>\big((2/\sqrt{3}){+}2\r_0\big){e^{2\k_0}}(\a e^2 \k_1)^{1/\a}/ (\pi \r_0^2). 
\]
\end{prop}
If $\ph$ is bounded above, so that $\k_1=0$, we recover the existence result of \cite{DG}. 

\begin{proof} The proof is identical to that of Proposition~\ref{prop:del2size}. One simply has to note that $\d(\n)\le((2/\sqrt{3}){+}2\r_0\big)a$ when $\n\in\Del_3(\o)$ for some $\o\in\Gb$.\qed\end{proof}

\smallskip\noindent
\emph{\thesubsubsection.2  Shape-dependent triangle interactions}.
The shape of a Delaunay triangle can be captured by its angles.  So let $\b(\n)$ and $\g(\n)$ respectively denote the smallest and the largest interior angle of a triangle $\n$. We assume that $\ph$ has the form
\be{del3:angle}
\ph(\n):=\phi(\b(\n),\g(\n))\,.
\ee

\begin{prop}\label{prop:del3angle}
Let $d=2$ and $\ph$ be a pure triangle potential on $\Del_3$ bounded below and satisfying \eqref{del3:angle}.
Suppose there exists some $\d>0$ such that 
$$\sup_{\g\ge\b>(\pi/3)-\d} \phi(\b,\g) < \infty\,.$$
Then there exists at least one Gibbs measure for $\ph$ and every $z>0$.
\end{prop}
\begin{proof}
We can again apply Corollary~\ref{cor3} with the assumption {\bf (U$^A$)}. The hypotheses {\bf (S)} and {\bf (R)} hold for the same reasons as in Proposition~{\ref{prop:del2size}}. The matrix $\sM$ and the set $\G^A$ can be defined as before, 
except that $b$ is now chosen so small that $\b(\n)>(\pi/3)-\d$ for all $\n\in\Del_3(\o)$ with $\o\in\Gb^A$.
Conditions (U1$^A$) and (U2$^A$) are  then obvious. 
By Remark~\ref{rem:scale-invariant}, it only remains to check (\^U3$^A$)
because $\ph$ is scale invariant. But this is trivial.
\qed\end{proof}

The conditions of Proposition~\ref{prop:del3angle} include the  \emph{hard-equilaterality model}, in which
\be{eq:equilaterality}
\phi(\b,\g)=\left\{\ba{cl}\infty&\text{if $\b\le(\pi/3)-\d$,}\\0&\text{otherwise}\ea\right.
\ee
for some $0<\d<\pi/3$. The associated Gibbs measure produces a random Delaunay triangulation for which each triangle is almost equilateral in the sense that every angle exceeds $\pi/3-\d$. This model is considered in \cite{DL2} (Model 1, Section~2.3) as the crystallized triangulation model.

\subsection{Neighborhood-dependent hyperedge interactions}

We now turn to examples of potentials $\ph$ for which $\ph(\n,\o)$ does not only depend on $\n$ but also on
the points of $\o$ in some neighborhood of $\n$. In two of these examples, the underlying hypergraph
structure is based on the singleton graph
\be{singleton}
\SG=\big\{(\{x\},\o): x\in\o\in\O\big\}
\ee
for which all hyperedges are singletons. Since this might seem trivial, we add immediately that more
complex hyperedges will be hidden in the horizons of the potentials.

\subsubsection{Forced-clustering $k$-nearest neighbor interactions}\label{sec:clustering}

In the $k$-nearest neighbors model on the space $\RR^d$ of any dimension $d$, each point $x\in\o$ interacts with the $k$
points of $\o$ that are closest to $x$.
This model was introduced in \cite{BBD99a,BBD99c}.  We focus here on a non-hereditary variant, 
in which the $k$-nearest neighbors of $x$ are forced to keep within distance $\d>0$ from each other.
This model was mentioned in \cite{DL} without precise justification of the proof, which is given here. 

Let $k\ge 1$ be some fixed positive integer. 
For each integer $1\le i\le k$, each configuration $\o\in\O$ with $\#(\o)\ge k{+}1$, and each $x\in\o$ we let 
$x^{i:\o}$ be the $i$-th nearest neighbor of $x$ in $\o$. More precisely, $x^{i:\o}$ is defined as 
the $i$-th element of $\o$ in the total order $<_{x}$ 
on $\RR^d$, in which  $y_1 <_x y_2$ if and only if
\[
|y_1-x|<|y_2-x| \mbox{ or } \big(|y_1-x|=|y_2-x| \mbox{ and } y_1<_d y_2\big). 
\]
Here, $<_d$ stands for the lexicographic order on $\RR^d$. We also set $x^{0:\o}=x$.

Now let
$\SG_k =\big\{(\{x\},\o)\in\SG: \#(\o)\ge k{+}1\big\}$.
We consider hyperedge potentials $\ph$ on $\SG_k$ of the form
\be{knng}  \ph(\{x\},\o)= \left\{ \begin{array}{cl}
 \phi(x^{0:\o},\ldots,x^{k:\o}) & \text{if } \diam\{x^{0:\o},\ldots,x^{k:\o}\} < \d, \\
 \infty & \text{otherwise,} 
 \end{array}\right.
\ee
where $\phi:(\RR^d)^{k+1}\to \RR$ is any \emph{bounded} measurable function  and $\d>0$ a fixed constant.
It is clear that this interaction is non-hereditary in the sense that
the removal of a particle from an allowed configuration can lead to a forbidden configuration.
Since no kind of symmetry is required of $\phi$, the present formalism includes both the directed and undirected $k$-nearest neighbor graphs.

\begin{prop}\label{nearest} In the forced-clustering $k$-nearest neighbor model described above,
there exists at least one Gibbs measure for $\ph$ and every $z>0$.
\end{prop}
\begin{proof}
Let us apply Theorem~\ref{th1}. The hypergraph  structure $\SG_k$ is clearly sublinear and $\ph$ is bounded below. 
The stability assumption {\bf (S)} thus follows from Remark~\ref{rem:linearity}. Next, for given
$(\{x\},\o)\in\SG_k$, the closed ball $\bar B(x,|x^{k:\o}-x|)$ with center $x$ and radius $|x^{k:\o}-x|$ can serve
as a horizon of $(\{x\},\o)$. 
Since the corresponding open ball contains at most $k$ points of $\o$, 
assumption {\bf (R)} follows immediately.  
Concerning {\bf (U)}, let $b>0$ be a number to be specified later
and $\sM=a\sE$ as in Remark \ref{uniformlocal} for some $a>2(b{+}\d)$. We set
$$ \G=\big\{ \o=\{x_0,x_1,\ldots,x_{k}\}:
x_0 \in B(0,b),\;x_i \in B(x_0,{\d}/{2} ) \ \ \forall\, i=1,\ldots,k
 \big\} .$$ 
For each $\o\in\Gb$ and $x\in\o$ it then follows that $|x-x^{k:\o}|<\d$. So (U1) holds with $r_\G=\d$, and 
(U2) follows with $c_\G^+\le (k{+}1)\|\phi\|_\infty$. Finally, we have
\begin{eqnarray*}
e^{z|C|}\,\Pi_C^z(\G) & = &  \frac{z^{k+1}}{(k{+}1)!} \int_C\ldots \int_C \1_\G(\{x_0,\ldots, x_{k}\})\, dx_0\ldots dx_{k} \\
& \ge & \frac{z^{k+1}}{(k{+}1)!}\, \nu_d^{k+1}\, b^d\left({\d}/{2}\right)^{kd}, 
\end{eqnarray*}
where $\nu_d$ is the volume of the unit ball in $\RR^d$. This shows that (U3) holds as soon as $b$ is chosen large enough.
\qed\end{proof} 

\subsubsection{Voronoi cell interactions}

Here we consider potentials that depend on the structure of the Voronoi cells. 
This type of model was introduced first by Ord; see the discussion in \cite{Ripley}.
The spatial dimension $d$ is arbitrary here. Specifically, let 
\[
\SG_b=\big\{(\{x\},\o)\in\SG: \Vor_\o(x) \text{ is bounded}\big\}
\]
be the  hypergraph structure of singletons with bounded Voronoi cells
and $\ph$ be of the form
\[
\ph(\{x\},\o)=\phi\big(\Vor_{\o}(x)\big)\ \text{ for }  (\{x\},\o)\in\SG_b.
\]
For instance, $\phi$ might depend on the number of faces, the volume, or the surface area of the Voronoi cells. 
The necessity of passing from $\SG$ to $\SG_b$ was discussed in Example \ref{ex:Delaunay}.
Each $(\{x\},\o)\in\SG_b$ has as finite horizon the bounded Voronoi flower
\begin{equation}\label{Delta}
\Delta=  \bigcup_{\xi \in  \Del(\o):\; \xi\ni x} \bar B(\xi,\o).
\end{equation}
The following proposition includes a result of \cite{BBD99a}.
\begin{prop}\label{del1:vor} In the Voronoi cell interaction model with bounded $\phi$,
there exists at least one Gibbs measure for $\ph$ and every $z>0$.
\end{prop}
\begin{proof} 
Let us apply Corollary~\ref{cor3} again. The stability comes from the sublinearity of $\SG_b(\o)$ and the lower boundedness of $\ph$, cf.~Remark \ref{rem:linearity}. The proof of {\bf (U$^A$)} is essentially the same as in Proposition~\ref{prop:del2size} with $\k_1=0$. Concerning the range condition~{\bf (R)}, let $\{x\}\in\SG_b(\o)$ and  $y_1$, $y_2$ any two points of the set 
$\Delta$ defined in (\ref{Delta}). Then there exist $\xi_1, \xi_2\in \Del(\o)$ such that $y_1\in \bar B(\xi_1,\o)$, $y_2\in \bar B(\xi_2,\o)$ and   $x\in\bar B(\xi_1,\o) \cap \bar B(\xi_2,\o)$. The latter means that the union of these two balls is connected.
By definition, $B(\xi_1,\o)$ and $B(\xi_2,\o)$ contain no point of $\o$. So, condition {\bf (R)} holds with $\ell_R=2$, $n_R=0$ and arbitrary $\delta_R$.\qed\end{proof}

In the above, the boundedness of $\phi$ was only assumed for simplicity. In analogy to Propositions~\ref{prop:del2size} and \ref{prop:del3size}, one can consider potentials $\ph$ that are polynomially increasing in the diameter of the cell's flower \eqref{Delta}. The existence result then holds only for sufficiently large activity 
$z$. It is also straightforward to consider hard-exclusion models as in Propositions~\ref{prop:del2rigid} and \ref{prop:del3angle}. 
For example, let $d=2$ and
\be{distortion}
\phi\big(\Vor_{\o}(x)\big)=\left\{\ba{cl}0&\text{if $\Vor_{\o}(x)$ has six edges (and vertices),}\\\infty&\text{otherwise.}\ea\right.
\ee
It is then clear that Theorem \ref{th1} applies, and by scale invariance a Gibbsian point process exists for all $z>0$. 
The typical configurations of such a Gibbs process preserve the topology of the triangular lattice but
are less regular than those of the hard-equilaterality model \eqref{eq:equilaterality} for small $\d$. 
The model \eqref{distortion} may therefore be called the \emph{randomly distorted triangular lattice}.

\subsubsection{Adjacent Voronoi cell interactions}

Here we reconsider the potential presented in Example \ref{ex:Delaunay},
which describes an interaction between two adjacent Voronoi cells.
That is, let $d=2$ and the set of hyperedges of a configuration $\o\in\O$ be given by 
\[
\Del_{2,b}(\o):=\big\{ \{x,y\}\in\Del_2(\o): \Vor_\o(x) \text{ and } \Vor_\o(y)\text{ are bounded}\big\}.
\]
Suppose the potential $\ph$ has the form
\[
\ph(\n,\o)=\phi\big(\Vor_{\o}(x),\Vor_{\o}(y)\big) \ \text{ for $\n=\{x,y\}\in\Del_{2,b}(\o)$. }
\]
For instance, $\ph(\n,\o)$ can either depend on the length of the common edge or on the area ratio of the cells $\Vor_{\o}(x)$ and $\Vor_{\o}(y)$.
As noticed before, the Voronoi ``doubleflower''
\[
\Delta=  \bigcup_{\xi \in  \Del(\o),\; \xi\cap\n\neq \emptyset } \bar B(\xi,\o).
\]
then serves as finite horizon of $(\n,\o)\in\Del_{2,b}$.
In this setting we have the following result which can be proved in the same way as Proposition~\ref{del1:vor}.

\begin{prop}\label{del2:vor} In the adjacent Voronoi cell interaction model in two dimensions with bounded $\phi$,
there exists at least one Gibbs measure for $\ph$ and every $z>0$.
\end{prop}

As in the case of Proposition~\ref{del1:vor}, 
this example can easily be extended to the case when $\phi$ is polynomially increasing in the diameter of
the associated doubleflower, or when $\phi$
exhibits a hard exclusion that permits the configurations in $\Gb^A$ for a suitable choice of $A$. A particular adjacent Voronoi cell interaction with hard exclusion was proposed in \cite{DL2}, Model 3, Section~2.3. In this model, a hard exclusion forces the cells to be 
neither too small nor too large, and a smooth contribution induces a competition between the areas of adjacent cells.

\subsubsection{Conclusion}

The preceding series of examples presents only a selection of possible models and could easily be extended. 
For instance, having dealt with interactions acting on single Voronoi cells or pairs of adjacent Voronoi cells, we could proceed to triples of Voronoi cells with a common point, which are indexed by Delaunay triangles, or even larger
clusters of Voronoi cells. On the other hand, the preceding interactions can be combined (i.e., added up) to obtain
models with a richer interaction structure. A further universe of models opens up if one passes to marked configurations
as in Remark \ref{rem:marked}.
Future will show which kind of interaction will prove suitable for modeling
geometric structures in the plane or in space that occur in the sciences. In any case, it should be evident that the conditions of our theorems
are flexible enough to guarantee the existence of Gibbsian point processes for a large variety of geometric interactions.

\section{Proofs of the theorems}\label{sec-proofs}

Let $\HE$, $\ph$ and $z$ be fixed throughout this section. 
Before we enter into the proofs of the theorems, it is necessary to verify a basic ingredient of the theory of Gibbs measures, namely the consistency of the finite-volume Gibbs distributions. 

\begin{lem}\label{consistency}
Let $\L\subset\D\bsub$ and $\o\in\O_*^{\D,z}$. Then
\[
G_{\D,\o}^z\big(\O_*^{\L,z}\big)=1\quad\text{ and }\quad
\int f\,dG_{\D,\o}^z = \int \big( \int f\,dG_{\L,\tilde\o}^z \big)\; G_{\D,\o}^z(d\tilde\o)
\]
for all measurable functions $f:\O\to[0,\infty \ro$.
\end{lem}

\begin{proof} Let $\L$, $\D$ and $\o$ be fixed. Since $G_{\D,\o}^z$ does not depend on $\o_\D$, we can assume for notational convenience that $\o\subset \D^c$. Consider any two configurations $\z\in\O_\L$ and 
$\xi\in\O_{\D\setminus\L}$.
It follows straight from the definition that $\HE_\L(\z\cup\xi\cup\o)\subset\HE_\D(\z\cup\xi\cup\o)$.
Hence, $H_{\D,\o}(\z\cup\xi)$ is the sum of $H_{\L,\xi\cup\o}(\z)$ and a term in $\RR\cup\{\infty\}$
which does not depend on~$\z$.  So, 
$H_{\L,\xi\cup\o}^-(\z)<\infty$ when $H_{\D,\o}^-(\z\cup\xi)<\infty$, and  
$H_{\L,\xi\cup\o}(\z)<\infty$ when $H_{\D,\o}(\z\cup\xi)<\infty$.
On the other hand, it follows that
\[
H_{\D,\o}(\z\cup\xi)+ H_{\L,\xi\cup\o}(\z') = H_{\D,\o}(\z'\cup\xi)+ H_{\L,\xi\cup\o}(\z)
\]
for all $\z'\in\O_\L$. Taking the negative exponential and integrating over $\z'$ we thus find that
\be{eq:consistent}
e^{-H_{\D,\o}(\z\cup\xi)}\, Z_{\L,\xi\cup\o}^z = e^{-H_{\L,\xi\cup\o}(\z)} \, Z_{\L,\D,\o}^z(\xi)\,,
\ee
where $Z_{\L,\D,\o}^z(\xi) = \int e^{-H_{\D,\o}(\z'\cup\xi)}\,\Pi_\L^z(d\z')$ is the partial partition function 
for which the configuration $\xi$ is held fixed.
Since $G_{\D,\o}^z$ is concentrated on the set $\{H_{\D,\o}^-<\infty, H_{\D,\o}<\infty\}$ which is 
contained in $\{H_{\L,\,\cdot\,\cup\o}^-<\infty,\ H_{\L,\,\cdot\,\cup\o}<\infty\}$, we can conclude that
\[
\begin{split}
G_{\D,\o}^z\big(Z_{\L,\,\cdot}^z=0 \big) &=
G_{\D,\o}^z\circ\pr_{\D\setminus\L}\inv\big( Z_{\L,\D,\o}^z(\cdot)=0\big)\\
& =(Z_{\D,\o}^z)\inv\int Z_{\L,\D,\o}^z(\cdot) \,\1_{\{Z_{\L,\D,\o}^z(\cdot)=0\}}\,d\Pi_{\D\setminus\L}^z=0\,.
\end{split}
\]
Likewise, $G_{\D,\o}^z\big(Z_{\L,\,\cdot}^z=\infty \big)=0$ because 
$\int Z_{\L,\D,\o}^z(\cdot) \,d\Pi_{\D\setminus\L}^z=Z_{\L,\o}^z<\infty$.
Combining these results we obtain the first assertion of the lemma. As a consequence, we can divide Eq.~\eqref{eq:consistent}
by $Z_{\D,\o}^z\,Z_{\L,\xi\cup\o}^z$ to obtain the consistency equation 
$G_{\D,\o}^z =\int G_{\L,\tilde\o}^z\, G_{\D,\o}^z(d\tilde\o)$.\end{proof}

We now turn to the proof of Theorem \ref{th1}. So we assume 
that the hypotheses {\bf (S)},  {\bf (R)} and  {\bf (U)} are satisfied. As usual, we construct a Gibbs measure
as a limit of Gibbs distributions in suitable boxes. We choose $\sM$ and $\G$ as in hypothesis {\bf(U)}
and consider for each $n\ge 1$ the parallelotope
\[
\L_n=  \bigcup_{k \in L_n}  C(k)\,,
\]
where $L_n=\{-n,\ldots,n\}^d$.
Let $\out\in\Gb$ be a fixed pseudo-periodic configuration with $\sup_{k\in \ZZ^d}N_{C(k)}(\out)<\infty$; 
the last condition can be satisfied by letting $\out$ be periodic. Assumption (U1) implies that $\out\in\Ocr^{\L_n}$.
Combined with \eqref{eq:Hfr}, {\bf (S)} and (U2\&3), this shows that
$\out$ is admissible for $\L_n$ and $z$; cf.\ \eqref{eq:H_upper_bound} below. So we can define the Gibbs distribution
\[
G_n= G_{\L_n,\out}^z\circ \pr_{\L_n}\inv
\]
in $\L_n$ with boundary condition $\out$ and activity $z$, projected to $\L_n$.
Since we aim at constructing a shift-invariant Gibbs measure, we will introduce a spatial averaging of $G_n$,
and it is convenient to work directly on the set of all shift-invariant probability measures on $\O$.

So, let $P_n$ be the probability measure on $(\O,\cF)$ relative to which the configurations in the disjoint blocks 
$\L_n+(2n{+}1)\sM k, k\in \ZZ^d$, are independent with identical distribution $G_n$. 
We consider the averaged measure
\be{eq:hat-P_n}
\hat P_n=\frac{1}{v_n} \int_{\L_n} P_n\circ\th^{-1}_x\, dx,
\ee
where $v_n=|\L_n|$ is the volume of $\L_n$. By the periodicity of $P_n$, $\hat P_n$ is shift-invariant.
Moreover, the intensity $i(\hat P_n)=\int N_{\L_n}\,dG_n/v_n$ of $\hat P_n$ is finite because
\[
\int N_{\L_n}\,e^{-H_{\L_n,\out}}\,d\Pi_{\L_n}^z \le  
e^{c_S\, \#(\bout[\L_n]{\out})}\int N_{\L_n}\,e^{c_S N_{\L_n}}\,d\Pi_{\L_n}^z <\infty
\]
by \eqref{eq:Hfr} and {\bf(S)}. So, $\hat P_n\in\sP_\Th$.

We will show that the sequence $(\hat P_n)$ has an accumulation point in a suitable topology.
As in \cite{GH}, we will take the required compactness from the compactness of the level sets
of the specific entropy. Let us recall the necessary concepts.

A measurable function $f:\O\to\RR$ is called local and tame if 
\[
f(\o)=f(\o_{\L}) \qquad \text{and} \qquad |f(\o)| \le a\,N_\L(\o)+b
\]
for all $\o\in\O$, some $\L\bsub$ and suitable constants $a,b\ge 0$. 
Let $\cL$ be the set of all local and tame functions. The
\emph{topology of local convergence}, or $\cL$-topology, on $\sP_\Th$ is then defined as the weak* topology induced by $\cL$, i.e., as the smallest topology for which the mappings $P \mapsto \int f \,dP$ (with $f \in \cL$) are continuous.
Note that the intensity $P\mapsto i(P)$ is continuous in the $\cL$-topology.

Next, for any $P\in\sP_\Th$ let $P_{\L_n}=P\circ \pr_{\L_n}\inv$ be the projection of $P$ to $\O_{\L_n}$ and
\[
 I(P_{\L_n}|\Pi^z_{\L_n})=\left\{
\begin{array}{ll}
\int f\ln f \, d\Pi^z_{\L_n} & \text{ if } P_{\L_n} \ll \Pi^z_{\L_n} \text{ with density }f , \\
\infty & \text{ otherwise}
\end{array}
\right.
\]
the relative entropy of $P_{\L_n}$ with respect to $\Pi^z_{\L_n}$; here, ``$\ll$'' stands for absolute continuity.
The \emph{specific entropy of $P$} (relative to $\Pi^z$) is then defined by
\begin{equation}\label{entropy}
I^z(P)=\lim_{n\to\infty} v_n\inv\, I (P_{\L_n}|\Pi^z_{\L_n});
\end{equation}
see \cite{GiiGibbs} and \cite{GiiZess} for the existence of the limit and further properties of $I^z$.
Our key tool is the following result of \cite[Lemma 3.4]{GiiPTRF}, which is based on \cite[Prop.~2.6]{GiiZess}.

\begin{prop}\label{compacite}
For all $c_1,c_2\ge 0$ and $z>0$, the set 
\[ 
\big\{ P\in\sP_\Th : I^z(P)-c_1\; i(P) \le c_2\big\}
\]
is relatively sequentially compact in the $\cL$-topology.
\end{prop}

In view of this fact, the following entropy bound implies that the sequence $(\hat P_n)$ has
a convergent subsequence.

\begin{prop}\label{entropybound} In the limit $n\to\infty$ we have
\be{se}
 I^{z}(\hat P_n) - c_S\; i(\hat P_n)\le |C|\inv\big(c_\G -\ln\Pi^z_C(\G)\big) + o(1).
\ee
\end{prop}
\begin{proof}
First of all, the definition of $\hat P_n$ readily implies that 
\be{eq:IP_n}
I^z(\hat P_n)=v_n\inv \,I(G_n|\Pi^{z}_{\L_n}) \;;
\ee
see the proof of \cite[Proposition~(16.34)]{GiiGibbs}. Likewise, $i(\hat P_n)=v_n\inv  \int  N_{\L_n}\, dG_n$.
By the definition of $G_n$, we know further that
\be{er}
I(G_n|\Pi^z_{\L_n})= -\int H_{\L_n,\out} \,dG_n - \ln Z_{\L_n,\out}^z\,.
\ee
So we need to estimate the two terms on the right-hand side.
As for the first term, hypotheses {\bf(S)} and (U1) give
\be{er2}
\int H_{\L_n,\out}\, dG_n \ge -c_S \int  N_{\L_n}\, dG_n-c_S \, \#(\boutS[\L_n]{\G}{\out}) \,,
\ee
and the assumption on $\out$ implies that $\#(\boutS[\L_n]{\G}{\out})=o(v_n)$.

It remains to estimate the partition function $ Z_{\L_n,\out}$. By (U1) we can find a number $m\ge1$
such that $\partial\L_n^\G\subset\L_{n+m}$ for all $n\ge1$. Fix any $n$ and
let $\z\in\O_{\L_n}$ be such that $\bar\z:= \z\cup\out_{\L_n^c} \in \Gb$. 
We claim that 
\be{eq:H_upper_bound}
H_{\L_n,\out}(\z) \le c_\G\, \# L_n + o(v_n)\,,
\ee
where the error term is uniform in $\z$. Indeed, since
$\n\in\HE_{\L_n}(\bar\z)$ when  $\n\in\HE(\bar\z)$ and $k\in\hat\n\cap L_n$, we can write 
\[
H_{\L_n,\out}(\z)
= \sum_{k\in L_n} \sum_{\n\in\HE(\bar\z):\,\hat\n\ni k} \frac{\ph(\n)}{\#\hat\n} 
+ \sum_{k\in L_{n+m}\setminus L_n}\, \sum_{\n\in\HE_{\L_n}(\bar\z):\,\hat\n\ni k} \frac{\ph(\n)}{\#\hat\n} \,.
\]
In view of (U2) and translation invariance, the first term on the right is not larger than $c_\G\, \# L_n$.
Likewise, the second term is dominated by $c_\G^+\, \#( L_{n+m}\setminus L_n)$. This proves~\eqref{eq:H_upper_bound} 
and leads us to the estimate
\be{er3}
Z_{\L_n,\out}^z \ge \int \1_{\Gb}(\bar\z)\, e^{-H_{\L_n,\out}(\z)}\,\Pi_{\L_n}^z(d\z)
\ge  e^{-c_\G\, \# L_n - o(v_n)} \ \Pi^z_C(\G)^{ \# L_n}.
\ee
Combining this with \eqref{eq:IP_n} to \eqref{er2} we end up with \eqref{se}.\end{proof}

The two propositions above imply that the sequence $(\hat P_n)$ admits a subsequence that converges to
some $\hat P\in\sP_\Th$ in the $\cL$-topology. The limit $\hat P$ is non-degenerate, in that $\hat P\ne\d_\emptyset$. Indeed, in view of
the lower semicontinuity of $I^z$ (implied by Proposition~\ref{compacite}) and the continuity of the intensity $i$
we obtain from \eqref{se} that
\be{eq:entropy_bound}
 I^{z}(\hat P) - c_S\; i(\hat P)\le |C|\inv\big(c_\G -\ln\Pi^z_C(\G)\big)\,.
\ee
But hypothesis (U3) ensures that the quantity on the right-hand side is strictly less than
$z= I^{z}(\d_\emptyset) - c_S\; i(\d_\emptyset)$.

It is natural to expect that $\hat P$ is the Gibbs measure we are looking for.
Unfortunately, however, we are unable to show that $\hat P$ is concentrated on the admissible configurations. 
However, since $\hat P$ is non-degenerate, we can consider the conditioned measure $P=\hat P(\,\cdot\,|\{\emptyset\}^c)\in \sP_\Th$ with $P(\{\emptyset\})=0$ and apply Proposition~\ref{prop: range-confinement}. Let us give a more precise statement of this proposition. 

Let $\ell_R, n_R,  \d_R$ be the constants
introduced in condition {\bf(R)}. Also, let $\d_-$ and $\d_+$ be the diameters of the largest open ball in $C$ and
of the smallest closed ball containing $C$, respectively. Fix an integer $m\ge 6\ell_R\d_+/\d_-$. For each $n\ge1$,
we decompose the parallelotope $\hat\L_n:=\L_{n+(2n+1)m}$
into the $(2m{+}1)^d$ translates $\L_n^k:= \L_n+(2n{+}1)\sM k$ of $\L_n$, where $k\in L_m$.
For any $\L\bsub$ let $n_\L\ge1$ be the smallest number with
$\L_{n_\L}\supset\L$ and $n_\L\ge \d_R/6\d_+$. For all $n\ge n_\L$ we consider the events
\be{eq:hat_O}
\hOcr^{\L,n} = \big\{ \min_{0\ne k\in L_m} N_{\L_n^k }> n_R \big\}\in \cF_{\hat\L_n\setminus\L}
\ee
as well as $\hOcr^\L=\bigcup_{n\ge n_\L}\hOcr^{\L,n}\in\cF_{\L^c}$.
We then have the following result.

\begin{prop}\label{Ploc}
Given any $\L\bsub$, we have $\hOcr^{\L}\subset\Ocr^\L$
and $\partial\L(\o)\subset\hat\L_n$ when $\o\in \hOcr^{\L,n}$ for some $n\ge n_\L$. Moreover,
$P(\hOcr^\L)=1$ for all $P\in\sP_\Th$ with $P(\{\emptyset\})=0$.
\end{prop}
\begin{proof}
Let $\L$ and $n\ge n_\L$ be fixed and consider any $\o\in\O$ for which the range of $\ph$ from $\L$ 
is not confined within $\hat\L_n$. Then there exists a configuration 
$\tilde\o\in\O$ with  $\tilde\o=\o$ on $\hat\L_n\setminus\L$, a configuration $\z\in\O_\L$ and
a hyperedge $\n\in\HE_\L(\z\cup\o_{\L^c})$ such that
$\ph(\n,\z\cup\o_{\L^c})\ne\ph(\n,\z\cup\tilde\o_{\L^c})$. So, every horizon $\D$ of $(\n,\z\cup\o_{\L^c})$
as in \eqref{eq:fhp} hits $\hat\L_n^c$. By \eqref{eq:HE_L}, $\D$ hits $\L$ too. 
Now let $\D$ be chosen as in {\bf(R)}.
We pick some $x\in \L\cap \D$ and $y\in \D\setminus\hat\L_n$ and,  as in {\bf(R)}, a chain 
of $\ell\le\ell_R$ balls that hit each other successively and run from $x$ to $y$. There is a first ball
$B_{\ell'}$ hitting $\hat\L_n^c$. Shrinking $B_{\ell'}$ if necessary, we find a connected chain
$B_1,\ldots,B_{\ell'}$ of at most $\ell_R$ balls
that are all contained in $\hat\L_n$, connect $x$ to $\partial\hat\L_n$, and have either diameter at most $\d_R$
or contain at most $n_R$ particles of $\z\cup\o_{\L^c}$. 
Since the distance between $x$ and $\partial\hat\L_n$ is at least $m(2n{+}1)\d_-$, at least one ball $B_i$ has a diameter 
exceeding $3(2n{+}1)\d_+$. As $n\ge \d_R/6\d_+$, this bound is larger than both $\d_R$ and 
$3\,\text{diam}\,\L_n$. So there exist at least two indices $k,k'\in L_m$
such that $\L_n^k$ and $\L_n^{k'}$ are included in $B_i$ and thus hold at most $n_R$ points of
$\z\cup\o_{\L^c}$. At least one of these 
parallelotopes is different from $\L_n$, say $\L_{n}^k$. 
This proves that $\o\notin\hOcr^{\L,n}$ and completes the proof of the first statement.
 
To prove the second claim let $P\in\sP_\Th$ be such that $P(\{\emptyset\})=0$. 
We have
\[
1-P(\hOcr^\L) = P\Big(\bigcap_{n\ge n_\L} (\hOcr^{\L,n})^c\Big)
\le \inf_{n\ge n_\L} \sum_{0\ne k\in L_m} P\big(  N_{\L_n^k }\le n_R\big)\,.
\]
By translation invariance, the last expression is equal to
\[
(\# L_m -1)  \  \inf_{n\ge n_\L}  P\big(N_{\L_n}\le n_R\big)\,.
\]
But this term vanishes because
\[
P\big(N_{\L_n}\le n_R\big) \to P\big(N_{\RR^d}\le n_R\big) = P(\{\emptyset\})=0
\]
as $n\ti$. The next to last identity comes from the well-known fact \cite[6.1.3]{KMM}
that $P(0<N_{\RR^d}<\infty)=0$ when $P$ is translation invariant.
The proof is therefore complete.\end{proof}

The final step in the proof of Theorem \ref{th1} is as follows.

\begin{prop} The conditional probability $P=\hat P(\,\cdot\,|\{\emptyset\}^c)\in \sP_\Th$ is a Gibbs measure for $\HE$, $\ph$ and $z$.
\end{prop}

\begin{proof} Since $\hat P\in \sP_\Th$ with $\hat P(\{\emptyset\})<1$, $P$ is well-defined and belongs to $\sP_\Th$.
To show that $P$ is a Gibbs measure we fix some $\L\bsub$ and consider the sets $\hOcr^{\L,p}$ defined in
\eqref{eq:hat_O} for $p\ge n_\L$. We also set $\hOcr^{\L,\le p}=\bigcup_{n=n_\L}^p \hOcr^{\L,n}$.
It is sufficient to show that
\be{A2}
\int_{\hOcr^{\L,\le p}} f \,d\hat P = \int_{\hOcr^{\L,\le p}\cap\O_*^{\L,z}} f_\L\, d \hat P
\ee
whenever $f:\O\to[0,1] $ is a local function and $p\ge n_\L$ is so large that $f$ is $\cF_{\hat\L_p}$-measurable.
Here, $f_\L$ is defined by
$$ f_\L(\o) :=  \int f\, dG_{\L,\o}\,.$$ 
Indeed,  letting $p\ti$ and setting $f=1$ we then find that $\hat P(\hOcr^{\L}\cap\O_*^{\L,z})=\hat P(\hOcr^{\L})$
and thus $P(\O_*^{\L,z})=1$ by Proposition~\ref{Ploc}. For arbitrary $f$ we obtain further that
$P= \int G_{\L,\o}^z\,P(d\o)$. Since $\L$ is arbitrary,
this means that $P$ is a Gibbs measure.

To prove \eqref{A2} let $f$ and $p\ge n_\L$  be fixed.
It will be convenient to replace the sequence $(\hat P_n)$ introduced in \eqref{eq:hat-P_n} by an alternative 
sequence of  measures with the same 
limit $\hat P$. Suppose $n$ is so large that $\hat\L_p \subset \L_n$ and let
$$ \L_n^\circ=\{x\in \RR^d :  \hat\L_p+x \subset \L_n\}$$
be the ``$\hat\L_p$-interior'' of $\L_n$, which coincides with the closure of
$\L_{n-p-(2p+1)m}$.
We define the  (subprobability) measure 
\[
\bar G_n:=\frac{1}{v_n} \int_{\L_n^\circ} G_{\L_n,\out}^z\circ\th^{-1}_x\, dx
=\frac{1}{v_n} \int_{\L_n^\circ} G_{\L_n-x,\th_{x}\out}^z\, dx\,;
\]
the equality comes from the shift-invariance of $\ph$ and the symmetry of $ \L_n^\circ$.
The argument in  \cite[Lemma 5.7]{GiiZess} then shows that $\int f\,d\hat P_n -\int f\,d\bar G_n\to0$
for all $f\in\cL$. This means that $\hat P$ can also be viewed as an accumulation point of the 
sequence~$(\bar G_n)$. Now let $x\in\L_n^\circ$, so that $\hat\L_p\subset\L_n-x$.
Using the consistency lemma~\ref{consistency} and the fact that  
$\hOcr^{\L,\le p}\in\cF_{\hat\L_p\setminus\L}\subset\cF_{(\L_n-x)\setminus\L}$
we find
\[
\int_{\hOcr^{\L,\le p}}\,f \,dG_{\L_n-x,\th_x\out}^z
= \int_{\O_*^{\L,z}\cap\hOcr^{\L,\le p} }
	\big(\int\,f\,dG_{\L,\o}^z\big) \,G_{\L_n-x,\th_x\out}^z(d\o)\,,
\]
and averaging over $x$ yields
\be{eq:finiteGibbs}
\int_{\hOcr^{\L,\le p}}\, f\,d\bar G_n
= \int_{\O_*^{\L,z}\cap\hOcr^{\L,\le p} }\, f_\L\,d\bar G_n\,.
\ee
The integrand of the integral on the left is measurable with respect to $\cF_{\hat\L_p\setminus\L}$ and thus belongs
to $\cL$.  By (A.4) in the appendix, the integrand on the right of \eqref{eq:finiteGibbs} is measurable with respect 
to the universal completion  $\cF_{\hat\L_p\setminus\L}^*$ and thus can be squeezed between two functions in
$\cL$ which coincide $\hat P$-almost surely; cf.\ \cite{Cohn}, Proposition~2.2.3. So, \eqref{eq:finiteGibbs} gives
\eqref{A2} in the limit when $n$ runs through a subsequence for which $\bar G_n$
tends to $\hat P$ in the $\cL$-topology.\end{proof}

The proof of Theorem \ref{th2} requires only two minor observations.
First, we note that (\^ U1) and (\^ U2) together with  {\bf (R)} imply (U1). Indeed, 
let $\o\in\Gb$. Condition (\^ U2) then shows that $H_{\L_m,\o}(\o_{\L_m})<\infty$ for all $m\ge1$.
If $m$ is so large that  $av_m-b>n_R$ for the constant $n_R$ in {\bf (R)}, 
the lower density bound  (\^ U1) thus gives that each
translate $\L_m^k$, $k\in\ZZ^d$, contains more than $n_R$ points of $\o$. 
Hence, every ball with at most $n_R$ points has a diameter no larger than $2\,\text{diam}(\L_m)$. Invoking
the range condition {\bf (R)}, we can therefore conclude that (U1) holds with 
$r_\G\le\ell_R\,\max(\d_R, 2\,\text{diam}(\L_m))$.

Next we note that the non-rigidity condition in its strong form (U3) was only used below \eqref{eq:entropy_bound}
when we showed that the accumulation point $\hat P$ is non-degenerate; for all other purposes, the weak form
(\^U3) was sufficient. However, the non-degeneracy of $\hat P$ is trivial under (\^ U1) because
$i(\hat P_n)\ge a-b/v_n$ and thus $i(\hat P)\ge a>0$ by the continuity of $i$.
This completes the proof of Theorem~\ref{th2}.

\renewcommand{\thesection}{\Alph{section}}
\setcounter{section}{1}
\section*{Appendix: Measurability}\label{app:measurability}

Here we collect and prove the measurability properties we have used
and add a further comment on measurability. Let $\L\bsub$ be fixed.

\begin{ass}
$\HE_\L:=\{(\n,\o):\n\in\HE_\L(\o)\}\in (\cF_f\otimes\cF)^*$.
\end{ass}\noindent 
Indeed, consider  the measurable functions
$f_\L(\n,\z,\o)=(\n,\z\cup\o_{\L^c})$ and $g(\n,\z,\o)=(\n,\o)$ from $\O_f\times\O_\L\times\O$ to
$\O_f\times\O$. Since $\HE$ and $\ph$ are measurable by assumption, the event $\bar\HE_\L:= \{\ph\circ g\ne \ph\circ f_\L\}$ 
then belongs to $\cF_f\otimes\cF_\L\otimes\cF$, and 
$\HE_\L$ is equal to the projection image $g(\bar\HE_\L)$. 
Since $\cF_\L$ is known \cite{Kall,KMM} to be the Borel $\s$-algebra for a Polish topology
on $\O_\L$, 
one can apply Prop.~8.4.4 of \cite{Cohn} to conclude that $\HE_\L$ is universally measurable, as claimed. 

\begin{ass}The functions $(\z,\o)\to H_{\L,\o}(\z)$ and  $(\z,\o)\to H_{\L,\o}^-(\z)$ 
are measurable with respect  to $(\cF_\L'\otimes\cF_{\L^c})^*$.
\end{ass}\noindent
To see this, we observe first that the mapping
$f_\L$ in Claim A.1 is measurable from $\cF_f\otimes\cF_\L'\otimes\cF_{\L^c}$ to $\cF_f\otimes\cF$, 
and therefore also from $(\cF_f\otimes\cF_\L'\otimes\cF_{\L^c})^*$ to $(\cF_f\otimes\cF)^*$; see
 \cite{Cohn}, Lemma~8.4.6. In view of Claim A.1, this means that the indicator function
 $\1_{\HE_\L}(\n,\z\cup\o_{\L^c})$ is measurable with respect to $(\cF_f\otimes\cF_\L'\otimes\cF_{\L^c})^*$.
Given any probability measure on $\O_f\times\O_\L\times\O_{\L^c}$, we can therefore
squeeze this indicator function between two 
$\cF_f\otimes\cF_\L'\otimes\cF_{\L^c}$-measurable functions which coincide almost surely; cf. Proposition~2.2.3 
of \cite{Cohn}. Writing 
\[
H_{\L,\o}(\z)= \sum_{\n\subset\o}\1_{\HE_\L}(\n,\z\cup\o_{\L^c})\,\ph(\n,\z\cup\o_{\L^c})\,,
\]
applying Theorem 5.1.2 of \cite{KMM} repeatedly when the indicator function is replaced by one of the squeezing functions 
and using Proposition~2.2.3 of \cite{Cohn} in the converse direction we get the result.  

\begin{ass}  The partition function $\o\to Z_{\L,\o}^z$ is measurable with respect to $\cF_{\L^c}^*$,
$\O_*^{\L,z}\in\cF_{\L^c}^*$, and $G^z_{\L,\o}(F)$ is a probability kernel
from $(\O_*^{\L,z},\cF_{\L^c}^*|_{\O_*^{\L,z}})$ to $(\O,\cF)$. 
\end{ass}\noindent
Let $P$ be an arbitrary probability measure on $\cF_{\L^c}$. As in Claim A.2, we can squeeze the function 
$e^{-H_{\L,\o}(\z)}$ between two $\cF_{\L}'\otimes\cF_{\L^c}$-measurable functions which coincide
$\Pi_\L^z\otimes P$-almost surely. Integrating these functions over $\z$ with respect to $\Pi_\L^z$ 
we obtain two functions of $\o$, which squeeze  $Z_{\L,\o}^z$,
are $\cF_{\L^c}$-measurable by the measurability part of Fubini's theorem, and coincide  $P$-almost surely.
As $P$ was arbitrary, the first result follows. In the same way one finds that the function
$\o\to\Pi_\L^z(H_{\L,\o}^-<\infty)$ is  $\cF_{\L^c}^*$-measurable. Hence
\[
\O_*^{\L,z}=\big\{ \o\in\O:\Pi_\L^z(H_{\L,\o}^-<\infty)=1, \ 0< Z_{\L,\o}^z<\infty\big\}\in\cF_{\L^c}^*\,.
\]
One also finds that the integral in \eqref{eq:Gibbs_distr} depends $\cF_{\L^c}^*$-measurably on $\o$,
which proves the last statement.

\begin{ass} Let $p\ge n_\L$ be fixed and suppose condition {\bf(R)} holds. 
Claims A.2 and A.3 remain  valid with $\cF_{\hat\L_p\setminus\L}$ in place of $\cF_{\L^c}$
as soon as all quantities are restricted to the set $\hOcr^{\L,p}$ defined in \eqref{eq:hat_O}.
In particular,  $\O_*^{\L,z}\cap\hOcr^{\L,p}\in\cF_{\hat\L_p\setminus\L}^*$, and
$G^z_{\L,\o}(F)$ is a probability kernel
from $(\O_*^{\L,z}\cap\hOcr^{\L,p},\cF_{\hat\L_p\setminus\L}^*|_{\O_*^{\L,z}\cap\hOcr^{\L,p}})$ to $(\O,\cF)$.
\end{ass}\noindent
Indeed, by Proposition~\ref{Ploc} and \eqref{eq:Hfr} we have
\[
H_{\L,\o}(\z)=\sum_{\n\in\HE_\L(\z\cup \o_{\hat\L_p\setminus\L})} \ph(\n,\z\cup \o_{\hat\L_p\setminus\L})
\quad\text{ when }\o\in \hOcr^{\L,p}\,.
\]
The counterpart of Claim A.2 is therefore obvious, and the analog of Claim A.3 follows as before.

\begin{ass}  The use of universal measurability could be avoided by modifying the definition of~$\HE_\L$.
\end{ass}\noindent
Namely, $\HE_\L(\o)$ could be defined as the set of all $\n\in\HE(\o)$
for which either $\n\cap\L\ne\emptyset$ or 
\[
\Pi_\L^z\otimes \Pi_\L^z\Big((\z_1,\z_2)\in\O_\L^2:
\ph(\n,\z_1\cup\o_{\L^c})\ne\ph(\n,\z_2\cup\o_{\L^c})\Big)>0\,.
\]
Then $\HE_\L\in \cF_f\otimes\cF$ by the measurability part of Fubini's theorem, and Claims A.2, A.3
and A.4 would follow without the stars referring to universal extensions.
This modified definition, however, is less intuitive and destroys the simple monotonicity of $\HE_\L(\o)$ in $\L$ 
which was used in the proof of Lemma~\ref{consistency}. One can still show that the required monotonicity holds for 
$\Pi_{\D\setminus\L}^z$-almost all $\xi$, but this is more involved. It is also necessary
to redefine $\Ocr^{\L}$ to obtain Proposition~\ref{Ploc}.
We therefore decided to make use of universal measurability.

\end{document}